\documentclass[12pt]{article}

\usepackage{amsmath}
\usepackage{amsfonts}
\usepackage{amsthm}
\usepackage{amssymb}
\usepackage{verbatim}
\usepackage{graphicx}
\usepackage{epsfig}

\newtheorem{theorem}{Theorem}[section]
\newtheorem{corollary}[theorem]{Corollary}
\newtheorem{lemma}[theorem]{Lemma}
\newtheorem{proposition}[theorem]{Proposition}
\newtheorem{cor}{Corollary}[section]
\theoremstyle{remark}

\theoremstyle{remark}

\theoremstyle{definition}
\newtheorem{defn}{Definition}[section]

\theoremstyle{definition}

\newcommand{\bmat}{\left[\begin{array}}
\newcommand{\emat}{\end{array}\right]}
\newcommand{\bt}{\begin{thm}}
\newcommand{\et}{\end{thm}}
\newcommand{\bp}{\begin{proof}}
\newcommand{\ep}{\end{proof}}
\newcommand{\bprop}{\begin{prop}}
\newcommand{\eprop}{\end{prop}}
\newcommand{\bl}{\begin{lemma}}
\newcommand{\el}{\end{lemma}}
\newcommand{\bc}{\begin{cor}}
\newcommand{\ec}{\end{cor}}
\newcommand{\bd}{\begin{defn}}
\newcommand{\ed}{\end{defn}}

\DeclareMathOperator{\m}{mod}

\DeclareMathOperator{\ppn}{ppn}
\DeclareMathOperator{\ppc}{ppc}

\usepackage{authblk}

\begin{document}

\title{Prime Power and Prime Product Distance Graphs}

\author{Joshua D. Laison}
\author{Yumi Li}
\author{Jeffrey Schreiner-McGraw}
\author{Colin Starr}

\affil{Mathematics Department \\
Willamette University \\
900 State St. \\
Salem, OR 97301}

\maketitle

\begin{abstract}
A graph $G$ is a $k$-prime product distance graph if its vertices can be labeled with distinct integers such that for any two adjacent vertices, the difference of their labels is the product of at most $k$ primes.  A graph has prime product number $\ppn(G)=k$ if it is a $k$-prime product graph but not a $(k-1)$-prime product graph.
Similarly, $G$ is a prime $k$th-power graph (respectively, strict prime $k$th-power graph) if its vertices can be labeled with distinct integers such that for any two adjacent vertices, the difference of their labels is the $j$th power of a prime, for $j \leq k$ (respectively, the $k$th power of a prime exactly).

We prove that $\ppn(K_n) = \lceil \log_2(n)\rceil - 1$, and for a nonempty $k$-chromatic graph $G$, $\ppn(G) = \lceil \log_2(k)\rceil - 1$ or $\ppn(G) = \lceil \log_2(k)\rceil$.  We determine $\ppn(G)$ for all complete bipartite, 3-partite, and 4-partite graphs.  We prove that $K_n$ is a prime $k$th-power graph if and only if $n < 7$, and
we determine conditions on cycles and outerplanar graphs $G$ for which $G$ is a strict prime $k$th-power graph.

We find connections between prime product and prime power distance graphs and the Twin Prime Conjecture, the Green-Tao Theorem, and Fermat's Last Theorem.
\end{abstract}

\bigskip


\section{Introduction}

Following Laison, Starr, and Walker \cite{Laison13}, a (\textit{\textbf{finite}}) \textit{\textbf{prime distance graph}} is a graph $G$ for which there exists a one-to-one labeling of its vertices $L:V(G) \to \mathbb{Z}$ such that for any two adjacent vertices $u$ and $v$, the integer $|L(u)-L(v)|$ is prime.  We let $L(uv)=|L(u)-L(v)|$.  We call $L$ a \textit{\textbf{prime distance labeling}} of $G$, so $G$ is a prime distance graph if and only if it has a prime distance labeling.  Note that in a prime distance labeling, the labels on the vertices of $G$ must be distinct, but the labels on the edges need not be.  Also note that by this definition, $L(uv)$ may still be prime even if $uv$ is not an edge of $G$.

Eggleton, Erd{\H{o}}s, and Skilton introduced infinite prime distance graphs and the study of their chromatic numbers in 1985 \cite{Eggleton85,Eggleton86, Eggleton88,Eggleton90,Voigt94,Yegnanarayanan02}.
In 2013, Laison, Starr, and Walker proved that trees, cycles, and bipartite graphs are prime distance graphs and that Dutch windmill graphs (also called friendship graphs) and paper mill graphs are prime distance graphs if and only if the Twin Prime Conjecture and dePolignac's Conjecture are true, respectively.  They also started a characterization of prime distance circulant graphs \cite{Laison13}.

In this paper, we extend the definition of prime distance graphs to prime product distance graphs and prime power distance graphs, for which the labels on the edges are products of at most $k$ primes and $k$th powers of primes, respectively.

\section{Prime Product Distance Graphs}

Given a graph $G$ and a one-to-one labeling of its vertices $L:V(G) \to \mathbb{Z}$, we say that $L$ is a \textit{\textbf{$k$-prime product distance labeling}} of $G$ and $G$ is a \textit{\textbf{$k$-prime product graph}} if for any two adjacent vertices $u$ and $v$, the integer $|L(u)-L(v)|$ has at most $k$ (not necessarily distinct) prime factors.  We require $|L(u)-L(v)|>1$ for all distinct vertices $u$ and $v$ of $G$.  If $G$ is a $k$-prime product graph and not a $(k-1)$-prime product graph, then we say the \textit{\textbf{prime product number}} of $G$ is $\ppn(G)=k$.  Every finite graph has finite prime product number, since labeling with consecutive even integers satisfies $|L(u)-L(v)|>1$ for all distinct vertices $u$ and $v$.  Also, $\ppn(G)=1$ if and only if $G$ is a prime distance graph, and $\ppn(G)$ is a monotone graph invariant: if $H$ is a subgraph of $G$, then $\ppn(H)\le \ppn(G)$.  Furthermore, if a graph $G$ has a $k$-prime product distance labeling then it has a $(k+1)$-prime product distance labeling, given by multiplying the labels on every vertex by a prime not yet used.  So for every graph $G$, $G$ is an $m$-prime product graph for all $m \geq \ppn(G)$.  In this section we determine bounds on $\ppn(G)$ for all graphs.

Recall the following graph theoretic terms \cite{West96}.  A graph $G$ is \textit{\textbf{$k$-partite}} or \textit{\textbf{$k$-colorable}} if $V(G)$ can be partitioned into $k$ disjoint subsets such that if two vertices of $G$ are adjacent, then they are in differerent subsets.  If every two vertices in different subsets are adjacent, $G$ is \textit{\textbf{complete $k$-partite}}.  A graph $G$ is \textit{\textbf{$k$-chromatic}} if $G$ is $k$-partite and not $(k-1)$-partite, or equivalently, if $k$ is the smallest number of colors needed to color the vertices of $G$ so that no two adjacent vertices have the same color.  Note that a complete $k$-partite graph is $k$-chromatic.

The following lemma generalizes a result of Eggleton, Erd{\H{o}}s, and Skilton \cite[Lemma~6]{Eggleton85}.

\begin{lemma} \label{colorable-lemma}
If $G$ has a $k$-prime product distance labeling, then $G$ is $2^{k+1}$-colorable.
\end{lemma}

\begin{proof}
Suppose $G$ has a $k$-prime product distance labeling $L$.  Color the vertices of $G$ with the integers $L(v) \m 2^{k+1}$, for $v \in V(G)$.  We verify that if two vertices $u$ and $v$ have the same color, then they are not adjacent.  If $u$ and $v$ have the same color then $L(u) \equiv L(v) \m 2^{k+1}$, so $2^{k+1}$ divides $|L(u)-L(v)|$.  The vertices $u$ and $v$ cannot be adjacent in $G$, since $L(uv)$ would be a product of at least $k+1$ primes.
\end{proof}

\begin{lemma}
For positive integers $m$ and $n$, if $m<n$ then $m$ has at most $\lceil \log_2(n)\rceil - 1$ (not necessarily distinct) factors. \label{factors}
\end{lemma}

\begin{proof}
By contrapositive, if $m$ has more than  $\lceil \log_2(n)\rceil - 1$ factors, then $m \geq 2^{\lceil \log_2(n) \rceil} \geq n$.
\end{proof}

\begin{theorem}
For $n \geq 3$, $\ppn(K_n) = \lceil \log_2(n)\rceil - 1$.
\end{theorem}
\begin{proof}
Let $k=\lceil \log_2(n)\rceil - 1$.  First we show that $K_n$ has a $k$-prime product distance labeling.  If the vertices of $K_n$ are $v_1$, $\ldots$, $v_n$, let $L(v_i)=2i$ if $i \leq n/2$ and $L(v_i)=2i+1$ if $i > n/2$.  Note that no two vertices are labeled with consecutive integers, so $|L(v_i)-L(v_j)|>1$ for all distinct vertices $v_i$ and $v_j$.  We verify that $|L(v_i)-L(v_j)|$ is the product of at most $k$ primes.  If $i$ and $j$ are both at most $n/2$, then $|L(v_i)-L(v_j)|=|2i-2j|<n$ has at most $k$ factors.
If  $i$ and $j$ are both greater than $n/2$, then $|L(v_i)-L(v_j)|=|(2i+1)-(2j+1)|<n$ has at most $k$ factors.
Finally, suppose $i > n/2$ and $j \leq n/2$.  Then $|L(v_i)-L(v_j)|=|2i+1-2j|$ is odd and less than $2n$.
So $|L(v_i)-L(v_j)|$ has at most $\lceil \log_3(2n)\rceil - 1 \leq k$ factors.


Conversely, by Lemma~\ref{colorable-lemma}, if $K_n$ has a $(k-1)$-prime product distance labeling, then $K_n$ is $2^k=(2^{\lceil \log_2(n)\rceil - 1})$-colorable, but $2^{\lceil \log_2(n)\rceil - 1}<n$, and $K_n$ is $n$-chromatic.

\end{proof}

Recall that the Green-Tao Theorem says that for any positive integer $j$, there exists a prime arithmetic progression of length $j$ \cite{green08}.  Theorem 1 of \cite{Laison13} uses the Green-Tao Theorem to construct a prime distance labeling of an arbitrary bipartite graph.  The following theorem generalizes this construction to construct a prime product labeling of an  arbitrary $k$-chromatic graph.

\begin{theorem} \label{k-partite}
For a nonempty $k$-chromatic graph $G$,  $\ppn(G) = \lceil \log_2(k)\rceil - 1$ or $\ppn(G) = \lceil \log_2(k)\rceil$.
\end{theorem}
\begin{proof}
By Lemma~\ref{colorable-lemma}, if $G$ has a $(\lceil \log_2(k)\rceil - 2)$-prime product distance labeling, then $G$ is $(2^{\lceil \log_2(k)\rceil - 2+1})$-colorable.  Since  $2^{\lceil \log_2(k)\rceil - 2+1}<k$, this contradicts the fact that $G$ is $k$-chromatic, so $\ppn(G) \geq \lceil \log_2(k)\rceil - 1$.

We show that $\ppn(G) \leq \lceil \log_2(k)\rceil$ by producing an $m$-prime product distance labeling of $G$ for a positive integer $m \leq \lceil \log_2(k)\rceil$.   Since $G$ is $k$-chromatic, $G$ can be partitioned into $k$ independent partite sets.  Let $p$, $p+d$, $\ldots$, $p+jk!d$ be a prime arithmetic progression for $p>2$ and where $j$ is the size of the largest partite set of $G$.  We label the $i_x$ vertices in the $x$th partite set with the labels $xp+k!d$, $xp+2k!d$, $\ldots$, $xp+i_xk!d$, and claim that this labeling is an $m$-prime product distance labeling for a positive integer $m \leq \lceil \log_2(k)\rceil$.

We first verify that these vertex labels are all distinct.  This is true for two vertices in the same partite set by construction.  Suppose that $xp+rk!d=yp+sk!d$ for some $x$, $y$, $r$, and $s$.  Then $(x-y)p=(s-r)k!d$.  Since $|x-y|<k,$ $(x-y)|k!,$ so $d|p$.  Since both $p$ and $p+d$ are prime, $\gcd(p,d)=1$, which implies that $d=1.$  However, since $p>2$ this implies $p+d=p+1$ is not prime, which is a contradiction.
Thus the vertex labels are all distinct.

Let $u$ and $v$ be two adjacent vertices in $G$ in the $x$th and $y$th partite sets, respectively.  We verify that $|L(u)-L(v)|$ has at most $\lceil \log_2(k)\rceil$ prime factors.  Suppose $L(u)=xp+rk!d$ and $L(v)=yp+s k!d$, and without loss of generality assume that $r>s.$  Then $|L(u)-L(v)|=|(xp+rk!d)-(yp+sk!d)|=|(x-y)|p+k!d(r-s)$.  Again since $|x-y|<k$, $(x-y) | k!$, so we can factor $|L(u)-L(v)|=|(x-y)|(p+ad)$ for some positive integer $a \leq jk!$.  By construction, $p+ad$ is prime.  Since $|x-y|<k$, $|x-y|$ has at most $\lceil \log_2(k)\rceil - 1$ factors by Lemma~\ref{factors}.  Thus $|L(u)-L(v)|$ has at most $\lceil \log_2(k)\rceil$ prime factors.
\end{proof}

By Theorem~\ref{k-partite}, for every $k$-chromatic graph $G$, there are two possibilities for $\ppn(G)$.  Note that both bounds are achieved: $C_3$ meets the lower bound and $C_4$ meets the upper bound.  We ask which $k$-chromatic graphs $G$ have $\ppn(G) = \lceil \log_2(k)\rceil - 1$ and which have $\ppn(G) = \lceil \log_2(k)\rceil$.  Note that if $\chi(G)=2$, then $G$ is a 1-prime product distance graph by Theorem 1 of \cite{Laison13}, so $\ppn(G)=\lceil \log_2(2)\rceil$.  If $G$ is 3-chromatic or 4-chromatic, then Theorem~\ref{k-partite} implies that $\ppn(G)=1$ or $\ppn(G)=2$.  Which 3-chromatic graphs and 4-chromatic graphs have $\ppn(G)=1$?  We answer this question here for complete 3-partite and complete 4-partite graphs.

In Lemma~\ref{K122-lemma} and Theorems~\ref{3-partite} and \ref{4-partite}, we use the following theorem of Laison, Starr, and Walker.  Recall that a graph $G$ is \textit{\textbf{2-odd}} if $G$ admits a labeling $L$ with the property that for distinct vertices $u$ and $v,$ $L(uv)$ is either odd or exactly 2.  Every prime distance graph is 2-odd.  For a 2-odd graph $G$, we color an edge $uv$ of $G$ red if $L(uv)=2$ and blue otherwise.

\begin{theorem}[{\cite[Theorem~13]{Laison13}}] \label{2-odd-characterization}
A graph $G$ is 2-odd if and only if it admits a red-blue edge-coloring satisfying the following two conditions:
\begin{enumerate}
    \item No vertex of $G$ has red-degree greater than 2. \label{red-degree}
    \item Every cycle in $G$ contains a positive even number of blue edges. \label{blue-cycle}
\end{enumerate}
\end{theorem}

We now consider complete 3-partite graphs $K_{a,b,c}$ for positive integers $a$, $b$, and $c$.  In what follows, assume $a \leq b \leq c$.

\begin{lemma}\label{K122-lemma}
Let $G=K_{1,2,2}$ be the complete 3-partite graph with partite sets $\{x\}$, $\{y_1, y_2\}$, and $\{z_1, z_2\}$, and suppose $L$ is a prime distance labeling of $G$.  If $L(x)=0$, then the remaining vertices have labels $2$, $-2$, $5$, and $-5$, with $L(y_1)=-L(y_2)$ and $L(z_1)=-L(z_2)$, and conversely.
\end{lemma}

\begin{proof}
Suppose $L(x)=0.$  Note that $L(y_1)$, $L(y_2)$, $L(z_1)$, and $L(z_2)$ are all prime or the negative of a prime since all are adjacent to $x$.  If $L(y_1)$, $L(y_2)$, $L(z_1)$, and $L(z_2)$ are all odd, then in a red-blue coloring of $G$, $y_1z_1y_2z_2$ is a red cycle, contradicting Theorem~\ref{2-odd-characterization}.  Without loss of generality, suppose $L(y_1)$ is even.  If $L(y_2)$, $L(z_1)$, and $L(z_2)$ are all odd, then $L(z_1y_2)=L(y_2z_2)=2$, so $L(y_2)$, $L(z_1)$, and $L(z_2)$ are 3, 5, and 7, or -3, -5, and -7.  Then $L(y_1z_1)$ or $L(y_1z_2)$ is 1 or 9, which is a contradiction.

Hence one of the numbers $L(y_2)$, $L(z_1)$, and $L(z_2)$ is even, but if one of $L(z_1)$ and $L(z_2)$ is even, say $L(z_1)$, then $xy_1z_1$ forms a red cycle, again contradicting Theorem~\ref{2-odd-characterization}.  Therefore $L(y_1)$ and $L(y_2)$ are both even, and since $y_1$ and $y_2$ are both adjacent to $x$, their labels must be $2$ and $-2$.  Without loss of generality, $L(y_1)=2$ and $L(y_2)=-2.$

Now the numbers $L(y_2z_1)$, $L(xz_1)$, and $L(y_1z_1)$ are all prime and of the form $p-2$, $p$, and $p+2$, so they must be 3, 5, and 7.  Thus $L(z_1)=L(xz_1)=5$ or $L(z_1)=-L(xz_1)=-5$, and the same is true of $L(z_2)$.

Conversely, suppose $L(y_1)=2$, $L(y_2)=-2$, $L(z_1)=5$, and $L(z_2)=-5$.  If $L(x)$ is odd, then $L(xz_1)=|L(x)-5|$ and $L(xz_2)=|L(x)+5|$ are both even and prime, so they are both 2, which is impossible.  If $L(x)$ is even, then similarly $L(xy_1)=|L(x)-2|$ and $L(xy_2)=|L(x)+2|$ are both 2, so $L(x)=0$.
\end{proof}

\begin{theorem}\label{3-partite}
For positive integers $a \leq b \leq c$,

\begin{enumerate}
    \item $\ppn(K_{1,1,c})=1$ if and only if there exist $c$ pairs of twin primes.
    \item $\ppn(K_{1,2,2})=1$.
    \item For all other values of $a$, $b$, and $c$, $\ppn(K_{a,b,c})=2$.
\end{enumerate}
\end{theorem}

\begin{proof}
(1.) Let $G$ be the complete 3-partite graph with partite sets $\{x\}$, $\{y\}$, and $\{z_1, \ldots, z_c\}$. Set $L(x)=0$, and note that $G$ consists of $c$ triangles, all with the edge $xy$.

Assume first that there exist $c$ twin prime pairs ($p_1$, $p_1+2$), $\ldots$, ($p_c$, $p_c+2$). Set $L(y)=2$ and $L(z_i)=p_i+2$.
This gives a prime distance labeling of $G$, so $\ppn(G)=1$.

Now assume that $\ppn(G)=1$.  Without loss of generality, we again set $L(x)=0$, and let $q=L(y)$.  If $q$ is odd, then in a red-blue coloring of $G$ either $xz_i$ or $yz_i$ must be red for each $i$, meaning that the corresponding edge label is 2. Since vertex labels are distinct, this can happen for at most four such edges, with corresponding vertex labels $L(z_i)=2$ or $L(z_i)=-2$ or $L(z_i)=q+2$ or $L(z_i)=q-2$. So $c \leq 4$ in this case, and we know there are $c$ pairs of twin primes.

If $q$ is even, we may take $q=2$.  Since $\ppn(G)=1$, $L(z_i)$ and $L(z_i)-2$ must both be prime for each $i$, so there exist $c$ pairs of twin primes.





\bigskip

(2.) A prime distance labeling of $K_{1,2,2}$ is given in Figure~\ref{K122}.  Note that this labeling is essentially unique by Lemma \ref{K122-lemma}.
\bigskip

(3.) We prove that $\ppn(K_{1,2,3})=2$ and $\ppn(K_{2,2,2})=2$.  Since every other complete 3-partite graph has one of these graphs as a subgraph, this would imply $\ppn(K_{a,b,c})=2$ for these graphs by Theorem~\ref{k-partite}.

First let $G=K_{1,2,3}$, and suppose $G$ has partite sets $\{x\}$, $\{y_1, y_2\}$, and $\{z_1,z_2,z_3\}$.  Suppose by way of contradiction that $L$ is a prime distance labeling of $G$.  By Lemma~\ref{K122-lemma}, we may assume that $L(x)=0$, $L(y_1)=2$, $L(y_2)=-2$, $L(z_1)=5$, and $L(z_2)=-5$.  Then, again by Lemma~\ref{K122-lemma}, $L(z_3)=5$ or $L(z_3)=-5$, which duplicates the label on $z_1$ or $z_2$, a contradiction.

Now let $G=K_{2,2,2}$, and suppose $G$ has partite sets $\{x_1,x_2\}$, $\{y_1, y_2\}$, and $\{z_1,z_2\}$.  Suppose by way of contradiction that $L$ is a prime distance labeling of $G$. Again by Lemma~\ref{K122-lemma}, we may assume that $L(x_1)=0$, $L(y_1)=2$, $L(y_2)=-2$, $L(z_1)=5$, and $L(z_2)=-5$.  Then, once more by Lemma~\ref{K122-lemma}, $L(x_2)=0$ also, a contradiction.
%
%

\begin{figure}
\begin{center}
\includegraphics[width=2in]{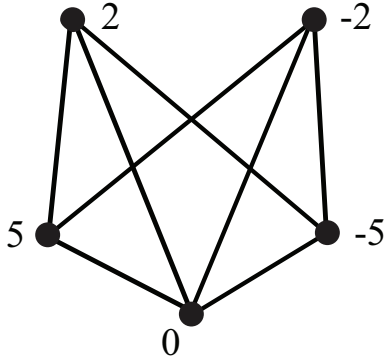} 
\caption{A prime distance labeling of $K_{1,2,2}$.} \label{K122} 
\end{center}
\end{figure}

\end{proof}

\begin{corollary}
We have $\ppn(K_{1,1,c})=1$ for all positive integers $c$ if and only if the Twin Prime Conjecture is true.
\end{corollary}

We now consider complete 4-partite graphs $K_{a,b,c,d}$ for positive integers $a$, $b$, $c$, and $d$.

\begin{theorem} \label{4-partite}
A 4-partite complete graph $G=K_{a,b,c,d}$ satisfies $\ppn(G) = 1$ if and only if $a=b=c=d=1$.
\end{theorem}

\begin{proof}
First note that $K_{1,1,1,1}=K_4$ is a prime distance graph with vertex labeling 0, 2, 3, 5.  We prove that $\ppn(K_{1,1,1,2})=2$.  Let $G=K_{1,1,1,2},$ and suppose $G$ has partite sets $\{x\}$, $\{y\}$, $\{z\}$, and $\{w_1, w_2\}$.

Suppose by way of contradiction that $L$ is a prime distance labeling of $G$ with $L(x)=0$. Let $G'$ be the graph obtained by deleting the edge $yz$ from $G$.  Note that $L$ is also a prime distance labeling of $G'$.  Since $G' \cong K_{1,2,2}$, by Lemma~\ref{K122-lemma}, $L(y)=2$ and $L(z)=-2$, or $L(y)=-2$ and $L(z)=2$, or $L(y)=5$ and $L(z)=-5$, or $L(y)=-5$ and $L(z)=5$.  In all cases, $L(yz)$ is not prime, a contradiction, so $\ppn(G)=2$.

%

Since $G$ is a subgraph of every complete 4-partite graph other than $K_{1,1,1,1}$, every other complete 4-partite graph has $\ppn(G)= 2$ by Theorem~\ref{k-partite}.


\end{proof}

\section{Prime Power Distance Graphs}

Given a positive integer $k$, we say that $G$ is a \textit{\textbf{prime $k$th-power distance graph}} if we can label the vertices of $G$ with distinct integers so that for every edge $uv$ of $G$, $|L(u)-L(v)|=p^j$ for some prime $p$ and some positive integer $j \leq k$.  We call $L$ a \textit{\textbf{prime $k$th-power distance labeling}} of $G$.  Note that the prime distance graphs are the prime first-power graphs.  Also observe that if $G$ is a prime $k$th-power graph, then so is every subgraph of $G$.  As is the case with $k$-prime product distance graphs, if $G$ is a prime $k$th-power distance graph, then $G$ is a prime $m$th-power distance graph for all $m \geq k$.


\begin{lemma}\label{K4-not-PPD}
In a prime $k$th-power distance labeling of $K_4$, the vertex labels cannot all have the same parity.
\end{lemma}
\begin{proof}
Suppose the vertices of $K_4$ are $a$, $b$, $c$, and $d$.  Let $L$ be a prime $k$th-power distance labeling of $K_4$, and suppose by way of contradiction that $L(a)$, $L(b)$, $L(c)$, and $L(d)$ have the same parity. Then the differences between these labels are all even, so each must be a power of 2.  Without loss of generality, $L(a)=0$, $L(b)=2^i$, $L(c)=\pm 2^j$, and $L(d)= \pm 2^k$, for positive integers $i$, $j$, and $k$ with $i \geq j$. We see that
$$|L(b)-L(c)| = |2^i \pm 2^j| = 2^j(2^{i-j} \pm 1).$$
This is not a prime power unless $i = j$, so $|L(b)-L(c)|=0$ or $|L(b)-L(c)|=2^{i+1}$. Since $L(c)\neq L(b)$ we must have $|L(b)-L(c)|=2^{i+1}$ and $L(c) = -2^i$.

Similarly, if $i\ge k,$ then $L(d)=-2^i=L(c)$, which is a contradiction.  Thus $k>i$.  Now $|L(b)-L(d)|=|2^i\pm 2^k|=2^i|1\pm 2^{k-i}|$. Since $k>i$, this is not a prime power, which is a contradiction.

Therefore, the vertices of $K_4$ cannot have four labels of the same parity.
\end{proof}

\begin{theorem}
There exists a positive integer $k$ such that $K_n$ is a prime $k$th-power distance graph if and only if $n < 7$.
\end{theorem}

\begin{proof}
If $n \geq 7$, by the Pigeonhole Principle any labeling of $K_n$ contains a $K_4$-subgraph in which  all vertices have the same parity. By Lemma \ref{K4-not-PPD}, this implies that $K_n$ has no prime $k$th-power distance labeling.

On the other hand, the complete graphs $K_n$ with $1 \leq n \leq 6$ have explicit prime $k$th-power distance labelings as follows.  The complete graph $K_4$ has a prime distance labeling  with labels $0$, $2$, $5$, and $7$, so $K_1$, $K_2$, $K_3$, and $K_4$ are prime distance graphs.  The complete graph $K_6$ has a prime-squared distance labeling with labels $0$, $2$, $4$, $7$, $9$, and $11$, so $K_5$ and $K_6$ are prime-squared graphs. (The graphs $K_5$ and $K_6$ are not prime distance graphs by \cite[Lemma~6]{Eggleton85}).

Thus $K_n$ is a prime $k$th-power distance graph if and only if $n< 7.$
\end{proof}

\subsection{Strict Prime Power Graphs}

Given a positive integer $k$, we say that $G$ is a \textit{\textbf{strict prime $k$th-power distance graph}} if we can label the vertices of $G$ with distinct integers so that for every edge $uv$ of $G$, $|L(u)-L(v)|=p^k$.  We call $L$ a \textit{\textbf{strict prime $k$th-power distance labeling}} of $G$.  In this case, it is not true that if $G$ is a strict prime $k$th-power distance graph, then $G$ is a strict prime $m$th-power distance graph for all $m \geq k$.  For example, we know that $K_3$ is a prime first-power distance graph, but we show in Proposition~\ref{C3} that $K_3$ is not a prime second-power distance graph.
In this section we investigate strict prime power labelings of cycles and outerplanar graphs.

\begin{proposition}
The graph $C_3 \cong K_3$ is not a strict prime $k$th-power graph for any $k\ge 2$. \label{C3}
\end{proposition}
\begin{proof}
Suppose by way of contradiction that $L$ is a strict prime $k$th-power distance labeling of $C_3$, with edges labeled $a^k$, $b^k$, and $c^k$, where $c>a,b$.  Then $a^k+b^k=c^k$.  If $k=2$, then $a^2+b^2=c^2$, and $a$, $b$, and $c$ are a Pythagorean triple.  But one of $a$, $b$, and $c$ is 2 by Lemma~\ref{2-odd-characterization}, and no such Pythagorean triple exists.  If $k>2$, then no such triple of numbers exists by Fermat's Last Theorem \cite{Cox94}.
\end{proof}

%

\begin{theorem}\label{even-cycles}
For all positive integers $k$ and $n$ with $n\ge 2$, the even cycle $C_{2n}$ is a strict prime $k$th-power graph.  Furthermore, we may pick the labels in a strict prime $k$th-power distance labeling of $C_{2n}$ to be arbitrarily large.
\end{theorem}

\begin{proof}
Suppose the vertices of $C_{2n}$ are $x_1$, $x_2$, $\ldots$, $x_{2n}$.  Let $p_1$, $\ldots$, $p_n$ be primes such that $\sum_{i=1}^{n-1} p_i^k < p_n^k$.  Label the vertices of $C_{2n}$ as follows:
$L(x_1)=0$, $L(x_j)=\sum_{i=1}^{j-1} p_i^k$ for $1 \leq j \leq n$, and  $L(x_j)=\sum_{i=j-n}^{n} p_i^k$ for $n+1 \leq j \leq 2n$.
By construction, the difference between any two consecutive labels is the power of a prime.  Also note that since we require $\sum_{i=1}^{n-1} p_i^k < p_n^k$, $\sum_{i=1}^{j-1} p_i^k \neq \sum_{i=m-n}^{n} p_i^k$ for any integers $j$ and $m$ with $1\le j\le n, n+1\le m\le 2n$, so vertex labels are distinct.  Finally, since $p_1$, $\ldots$, $p_n$ can be any primes satisfying this inequality, we may pick them to be arbitrarily large.
\end{proof}

\begin{theorem}\label{odd-cycles}
If an odd cycle $C_{2n+1}$ is a strict $k$th-power prime distance graph, then so is every larger odd cycle $C_{2(n+j)+1}$, $j>1$.
\end{theorem}
\begin{proof}
Suppose the vertices of $C_{2n+1}$ are $x_1$, $x_2$, $\ldots$, $x_{2n+1}$, and that $L$ is a strict $k$th-power prime distance labeling of $C_{2n+1}$ with $L(x_ix_{i+1})=p_i^k$ for some primes $p_i$, $1 \leq i \leq 2n$.  Without loss of generality we may assume $L(x_1)=0$, so $L(x_m)=\sum_{i=1}^{m-1} a_i p_i^k$, where $a_i= \pm 1$,  for $2 \leq m \leq 2n+1$, and $L(x_{2n+1})=\sum_{i=1}^{2n} a_i p_i^k$ must also be a prime $k$th power $p_{2n+1}^k$.


Suppose $C_{2n+2j+1}$ has vertices $z_1$, $z_2$, $\ldots$, $z_{2n+2j+1}$.  Choose primes $q_1$, $\ldots$, $q_j$, all larger than $|p_i|$ for all $1 \leq i \leq 2n+1$, and such that $q_j>\sum_{i=1}^{2n} p_i^k ++\sum_{i=1}^{j-1} q_i^k$.  We define a labeling $L'$ of $C_{2n+2j+1}$ by
$$L'(z_m)=\left\{
  \begin{array}{ll}
   0 & \text{ if } m=1, \\
   a_1 p_1^k & \text{ if } m=2, \\
   a_1 p_1^k+\sum_{i=1}^{m-2} q_i^k & \text{ if } 3 \leq m \leq j+2, \\
   \sum_{i=1}^{m-j-1} a_i p_i^k +\sum_{i=1}^j q_i^k & \text{ if } j+3 \leq m \leq j+2n+1, \\
   \sum_{i=m-j-2n-1}^j q_i^k & \text{ if } j+2n+2 \leq m \leq 2j+2n+1. \\
  \end{array}
\right. $$
By construction, the difference between any two consecutive labels is the power of a prime, and by our choice of the primes $q_i$, the labels on vertices are distinct.


\end{proof}

\begin{corollary} If $C$ is an odd cycle that is a strict $k$th-power prime distance graph, then so is every larger cycle.
\end{corollary}

\begin{theorem}\label{existence}
For all $k \geq 1$ there exists some odd $N$ such that $C_N$ is a strict prime $k$th-power distance graph.
\end{theorem}
\begin{proof}
Since $3^k$ and $5^k$ are relatively prime, there exist integers $r$ and $s$ such that $3^kr+5^ks=1$, $r<0$, and $s>0$.  Let $a=2^kr$ and $b=2^ks$.  Then $3^ka+5^kb=2^k$, $a<0$, and $b>0$.  We claim that $C_N$ is a strict prime $k$th-power graph if $N=b-a+1$.  Note that $a$ and $b$ are even, so $N$ is odd.  Suppose the vertices of $C_N$ are $x_1$, $x_2$, $\ldots$, $x_N$.  We define a strict prime $k$th-power distance labeling $L$ of $C_N$ by $L(x_1)=0$, $L(x_i)=(i-1) \cdot 5^k$ for $2 \leq i \leq b+1$, and $L(x_i)= (b+2-i-a) \cdot 3^k$ for $b+2 \leq i \leq b-a+1$.  For any pair of vertices $x_i$ and $x_{i+1}$, we have $|L(x_i)-L(x_{i+1})|=5^k$, $|L(x_i)-L(x_{i+1})|=3^k$, or $|L(x_i)-L(x_{i+1})|=2^k$ in the case $i=b+1$.
\end{proof}

We denote the least integer $N$ such that $C_n$ is a strict $k$th-power prime distance graph for all $n\ge N$ by \textbf{$\ppc(k)$}.  Theorem \ref{even-cycles}, Theorem \ref{odd-cycles}, and  Theorem \ref{existence} together imply that $\ppc(k)$ exists for all $k$.  By \cite[Theorem~3]{Laison13}, $\ppc(1)=3$.  Since $C_7$ can be $2$nd-power prime distance labeled with $0$, $4$, $3485$, $3124$, $2283$, $74$, and $25$, and $C_3$ cannot be $2$nd-power prime distance labeled by Proposition~\ref{C3}, $\ppc(2)=5$ or $\ppc(2)=7$.  We know no other values of $\ppc(k)$.

Recall that a graph $G$ is \textit{\textbf{outerplanar}} if $G$ can be drawn in the plane with no edge crossings and all vertices on the outside face.

\begin{lemma} \label{leaf-cycle}
If $G$ is a connected outerplanar graph with at least two cycles and no vertices of degree 1, then $G$ has a cycle $C$ with one or two vertices $x$ or $x$ and $y$, such that $G-x$ (respectively $G-\{x,y\}$) has at least two connected components, one of which is $C-x$ (respectively $C-\{x,y\}$).
\end{lemma}

\begin{proof}
The block-cutpoint graph of $G$ is a tree \cite{West96}, so we consider a leaf-block $B$ of $G$.  If $B$ is a single cycle, we may take $C=B$.  Otherwise, since $B$ is 2-connected, the weak dual $D$ of $B$ is a tree with at least two leaves \cite{West96}.  We consider the cycles of $B$ corresponding to the leaves of $D$.  There are at least two of these cycles.  If they all contain the cut vertex of $G$ in $B$, then $D$ is a path, and we may take the cycle of $B$ corresponding to either endpoint of $D$ as $C$.  Otherwise we take $C$ to be any cycle of $B$ corresponding to a leaf of $D$ that doesn't contain the cut vertex of $G$ in $B$.
\end{proof}

\begin{theorem}
Given a natural number $k$, if $G$ is an outerplanar graph with girth at least $\ppc(k)+6$, then $G$ is a strict $k$th-power prime distance graph. \label{outerplanar}
\end{theorem}
\begin{proof}

Suppose $G$ is an outerplanar graph with girth at least $\ppc(k)+6$. If $G$ is a single cycle, then $G$ is a strict $k$th-power prime distance graph by Theorems~\ref{even-cycles} and \ref{odd-cycles}.  If $G$ has a vertex $x$ of degree 1 with neighbor $y$, then by induction we may $k$th-power prime distance label $G-x$, and then label $xy$ with $p^k$ for any prime $p$ larger than the labels on $G-x$. If $G$ is not connected, we may $k$th-prime power distance label each connected component of $G$.


%
Therefore we may assume that $G$ is a connected outerplanar graph with at least two cycles and no vertices of degree 1.
By Lemma~\ref{leaf-cycle}, $G$ has a cycle $C$ with one or two vertices $x$ or $x$ and $y$, such that $G-x$ (respectively $G-\{x,y\}$) has at least two connected components, one of which is $C-x$ (respectively $C-\{x,y\}$).  In either case we have a vertex $x$ of $G$.  If we don't yet have a vertex $y$ of $G$, let $y$ be a vertex in $C$ incident to $x$.

By induction, let $L_1$ be a strict $k$th-power prime distance labeling of $G'=(G-C)\cup xy$ with $L_1(x)=0$ and $L_1(y)=p^k$ for some prime $p$.

Suppose $C$ has $m+6$ vertices.  Since the girth of $G$ is at least $\ppc(k)+6$, we have $m \geq \ppc(k)$.  Say the vertices of $C$ are $x$, $y$, $z$, $x_1$, $x_2$, $\ldots$, $x_{m-1}$, $x_m$, $a$, $b$, and $c$, in cyclic order.  Now consider a cycle $C_2$ with vertices $x_1$, $x_2$, $\ldots$, $x_m$.  Since $C_2$ has $m \geq \ppc(k)$ vertices, $C_2$ has a strict $k$th-power prime distance labeling $L_2$, and we may choose $L_2(x_1)=0$.

We now define a labeling of $G$, which we claim is a strict $k$th-power prime distance labeling.  Let $q$ and $r$ be primes such that $q^k$ and $r^k$ are larger than the sum of the absolute values of all labels in $L_1$ and $L_2$.  Suppose $L_1(xy)=p^k$ and $L_2(x_1x_2)=s^k$ for some primes $p$ and $s$.  For a vertex $u \in G-C$, define $L(u)=L_1(u)$.  For a vertex $x_i \in C$ with $1<i \leq m$, define $L(x_i)=L_2(x_i)+p^k+r^k+q^k$. Finally, define $L(z)=p^k+r^k$, $L(x_1)=p^k+r^k+s^k$, $L(a)=p^k+r^k+q^k$, $L(b)=r^k+q^k$, and $L(c)=r^k$.

We check that vertex labels in $L$ are distinct.  Since the labels on $C-\{x,y\}$ were chosen to be larger than the labels on $G-C$, labels on $C-\{x,y\}$ are distinct from labels on $G-C$.  Since labels on $G-C$ in $L$ are the same as labels on $G-C$ in $L_1$, they are distinct by induction, and likewise for the labels on $x_i \in C$ with $1<i \leq m$.  Again since $q^k$ and $r^k$ are larger than the sum of labels in $L_1$ and $L_2$, the labels on $a$, $b$, $c$, $z$, and $x_1$ are also distinct from each other and from other labels in $G$.

Finally we check that if $uv \in G$, $L(uv)$ is a $k$th power of a prime.  Again this is true by induction if $uv \in G-C$ or $uv \in C-\{a,b,c,x,y,z,x_1\}$.  We check the remaining edges directly: $L(xy)=p^k$, $L(yz)=r^k$, $L(zx_1)=s^k$, $L(x_1x_2)=q^k$, $L(x_ma)=L_2(x_1x_m)$, $L(ab)=p^k$, $L(bc)=q^k$, and $(xc)=r^k$.

\end{proof}

\section{Open Questions}

We conclude with a list of open questions.

\begin{enumerate}
\item The converse to Lemma~\ref{colorable-lemma} is false: in \cite{Laison13}, the authors gave an example of a 4-chromatic graph which is not a prime distance graph, i.e.~a graph which is $2^{k+1}$-colorable with no $k$-prime product labeling for $k=1$.  Do these graphs exist for other values of $k$?

\item Which 3-chromatic and 4-chromatic graphs $G$ have $\ppn(G)=1$?

\item Which 3-chromatic and 4-chromatic graphs $G$ have $\ppn(G)>2$?

\item By the remarks after Theorem~\ref{existence}, $\ppc(2)=5$ or $\ppc(2)=7$.  Which of these is true, and what is $\ppc(k)$ for any integer $k > 2$?

\item Theorem~\ref{outerplanar} describes a family of outerplanar strict $k$th-power prime distance graphs.  Which outerplanar graphs are strict $k$th-power prime distance graphs for some $k$?

\end{enumerate}

\bibliographystyle{plain}
\bibliography{prime-power-product-distance-graphs}

\end{document}